\documentclass{amsart}
\usepackage{graphicx}
\usepackage{newlattice}
\usepackage{enumerate}

\DeclareMathOperator{\Canon}{Canon}
\DeclareMathOperator{\Cov}{Cov}

\DeclareMathOperator{\Hi}{Hi}
\DeclareMathOperator{\imbal}{Imb}
\DeclareMathOperator{\LM}{LM}
		    
\DeclareMathOperator{\Left}{Left}
\DeclareMathOperator{\Lo}{Lo}
\DeclareMathOperator{\MC}{MC}
\DeclareMathOperator{\Mid}{Mid}			

\DeclareMathOperator{\Over}{Over}
\DeclareMathOperator{\Right}{Right}
\DeclareMathOperator{\RM}{RM}
\DeclareMathOperator{\LR}{LocReg}
\DeclareMathOperator{\Int}{LocInt}
\DeclareMathOperator{\Under}{Under}

\newcommand{\elo}{<_{\incomp}}
\newcommand{\eloi}{\prec_{\incomp}}
\newcommand{\geqx}{\geq_\textup{pf}}
\newcommand{\rgl}{\mathbin{\gl}}
\newcommand{\gtx}{>_\textup{pf}}

\newcommand{\incomp}{\parallel}
				
\newcommand{\Leftx}{\textup{Left}'}

\newcommand{\ltx}{<_\textup{pf}}
\newcommand{\leqx}{\leq_\textup{pf}}
\newcommand{\Midx}{\textup{Mid}'}     
				
\newcommand{\Rightx}{\textup{Right}'}
\newcommand{\rtheta}{\mathbin{\theta}}
\newcommand{\rgr}{\mathbin{\gr}}
\newcommand{\slo}{<_\textup{lr}}   
   
\newcommand{\rgt}{\mathbin{\gt}}	
\newcommand{\slr}{<_\textup{lr}}

\newcommand{\RR}{\mathbb{R}}

\newcommand{\Tur}{\textup{T}_\textup{UR}}		
\newcommand{\Tul}{\textup{T}_\textup{UL}}		
\newcommand{\preclr}{\prec_\textup{lr}}

\theoremstyle{plain}
\newtheorem{theorem}{Theorem}
\newtheorem{lemma}[theorem]{Lemma}
\newtheorem{corollary}[theorem]{Corollary}

\begin{document}

\title{Zilber's Theorem for planar lattices, revisited}

\author[Kirby A. Baker]{Kirby A. Baker}
\email[K.\,A. Baker]{Kirby Baker <baker@math.ucla.edu>} 

\author{George Gr\"{a}tzer} 
\email[G. Gr\"atzer]{gratzer@me.com}
\urladdr[G. Gr\"atzer]{http://server.maths.umanitoba.ca/homepages/gratzer/}
\dedicatory{To G\'abor Cz\'edli,\\
on his 65th birthday}

\date{December 26, 2019}

\begin{abstract}
Zilber's Theorem states that a finite lattice $L$ 
is planar if{}f it has a complementary order relation.
We provide a new proof for this crucial result
and discuss some applications,
including a canonical form for finite planar lattices 
and an analysis of coverings in the left-right order.
\end{abstract}

\maketitle

\section{Introduction}\label{S:intro}

\subsection{Zilber's Theorem}\label{S:Zilber-plus}

While there is a brief discussion of planar orders 
and lattices in G. Birkhoff~\cite{gB48, gB67}, 
the foundation of this field was laid thirty-two years later in
K.\,A. Baker, P.\,C. Fishburn, 
and F.\,S. Roberts \cite{BFR72},
closely followed by David Kelly and Ivan Rival \cite{KR75},
which has become the standard reference.  We revisit this
theory with a focus on Zilber's Theorem, its proof, and new
applications, which include a canonical form for finite planar
lattices and an analysis of coverings in the left-right order.

Several dozen articles have been published on planar lattices
since the publication of the original two papers.
The References include a partial list from 1972 to 2019.
Many basic questions have yet yo be answered; 
we intend to address some of them.

A (finite) ordered set $(P, <)$ is said to have complementary (strict) order relation~$\gl$ if any two distinct elements of $P$ are comparable by exactly one of $<$ and $\gl$.

\begin{theorem}[Zilber's Theorem]\label{T:Zilber}
A finite lattice $L$ is planar
if{}f it has a complementary order relation.
\end{theorem}

G. Birkhoff \cite{gB48, gB67} states this theorem as Exercise 7~(c)* in Section II.4. 
This result has immediate interesting ramifications.
Kelly-Rival~\cite{KR75} present
a proof of the ``only if'' implication.

\subsection{A little history}\label{S:history}

Joseph A. Zilber (1929--2009)
received his A.B (1943), A.M. (1946), and Ph.D. (1963) at Harvard. His Ph.D. thesis under
Raoul Bott was entitled ``Categories in Homotopy Theory''.
Earlier, he had published two papers (\cite{EZ53a} and 
\cite{EZ53}) with S. Eilenberg.
 
While a student, 
he probably communicated to Garrett Birkhoff the result 
we refer to as Zilber's Theorem,
when Birkhoff started to collect material for the revised edition of his book \emph{Lattice Theory} ~\cite{gB48}, published in 1948.

Zilber's Theorem appears to have been largely forgotten.
Despite the fact that there have been dozens of publications on planar lattices, we found only three references to~it, namely, 
K.\,A. Baker, P.\,C. Fishburn, 
and F.\,S. Roberts \cite[p. 18]{BFR72},
Kelly-Rival \cite[Proposition 1.7]{KR75}, 
and G.~Cz\'edli and G.~ Gr\"atzer 
\cite[Exercises 3.6 and 3.7]{CG14}.
On the other hand, the Mathematical Reviews 
lists 138 references to the closely related article of
B. Dushnik and E.\,W. Miller \cite{DM41},
see Section~\ref{S:DMv2}.

Kelly-Rival \cite{KR75} and David Kelly \cite{dK87}
take the theory 
further by considering coverings that are determined 
by curves rather than line segments, 
which utilizes additional technical tools. 
The present exposition is restricted to line segments only.

\subsection{Expanded form via the Dushnik-Miller Theorem}\label{S:DMv2}

B. Dushnik and E.\,W. Miller \cite[Theorem 3.61]{DM41} 
preceded Zilber by listing several conditions 
on an order relation that are 
equivalent to its having a complementary relation. 
Zilber's Theorem combined with these 
gives the following expanded form. 
(See also G. Birkhoff \cite{gB48, gB67}, K.\,A. Baker, P.\,C. Fishburn, 
and F.\,S. Roberts \cite{BFR72},
T. Hiraguti \cite{tH51}, \cite{tH55}, and 
Kelly-Rival \cite[Proposition 5.2]{KR75}.)

\begin{theorem}[Zilber's Theorem Plus]
\label{T:Zilber-plus}
For a finite lattice $L$, the following conditions are equivalent.
\begin{enumeratei}
\item $L$ is planar;
\item $L$ has a complementary order relation;
\item $L$ has order dimension at most $2$;
\item $L$ has an order embedding into the direct product of two chains.
\end{enumeratei}
\end{theorem}

Conditions (ii) and (iii) are those of Dushnik and Miller, while (iv) is a rewording of the Dushnik-Miller condition that $L$ is isomorphic to a lattice of intervals of a chain. The Dushnik-Miller conditions apply more generally to ordered sets but are stated here for lattices only, the context of Zilber's Theorem. The proof of Zilber's Theorem Plus (in Section~\ref{S:Zilber-proof}) is also a convenient route to the proof of Zilber's Theorem itself. It may be that any proof of Zilber's Theorem would involve the same conditions in some form.

\subsection{Left-right in planar lattices}\label{S:Left}
The theory of planar lattices 
is a mixture of lattice theory and geometry.
For $a \incomp b \in L$ (incomparable elements $a, b$)
of a planar lattice diagram $L$, we have to define 
the geometric concept of when $a$ is to the left of $b$.

We use the following geometric paradigm. 
In a planar lattice diagram $L$, each maximal chain $C$ has a path formed by the covering line segments in $C$, and this path divides $L$ into two halves overlapping at $C$, the left half and the right half.
Then for $\ell \incomp r \in L$, 
we define $\ell$ to be to the left of $r$,
if $\ell$ is on the left and $r$ 
is on the right side of some maximal chain $C$;
see Figure~\ref{F:left-right} for the three possibilities.

\begin{figure}[htb]
\centerline{\includegraphics{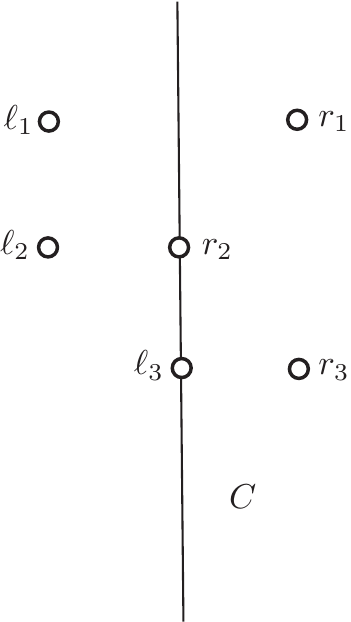}}
\caption{$\ell_i$ is to the left of $r_i$ for $i = 1,2,3$}
\label{F:left-right}
\end{figure}

\subsection{Outline}\label{S:Outline}

In Section~\ref{S:diagrams}, 
we define the basic diagram concept.
Section~\ref{S:Maximal} introduces
the distributive lattice of maximal chains of a planar lattice diagram.
Section~\ref{S:complementary} presents our approach to the construction of the complementary order in Zilber's Theorem Plus.
Section~\ref{S:R2} looks at $\RR^2$
with various orders.

Zilber's Theorem Plus is verified in Section~\ref{S:Zilber-proof}.
In Section~\ref{S:canonical},
we introduce the \emph{canonical form} of a 
planar lattice diagram, with its desirable properties.
Sections~\ref{S:left-right-coverings} and~\ref{S:traversals} give a more detailed picture of conditions (ii) and (iii) of Zilber's Theorem Plus. Section~\ref{S:left} discusses alternative constructions for the complementary order,
including that of Kelly-Rival \cite{KR75}. 
In~Section~\ref{S:Quotient} we
apply Zilber's Theorem to prove that 
quotients of planar lattices are also planar.

\subsection{Terminology and notation}\label{SS:Terminology}
In most of the literature, a ``planar lattice'' is a lattice that has a planar diagram, but there is not necessarily a formal distinction between the lattice and the diagram. In the present theory it is necessary to consider diagrams as formal objects in themselves.
G. Cz\'edli and G. Gr\"atzer \cite{CG14}
follow this route, as do Kelly-Rival \cite{KR75}. 
(For more recent papers utilizing
this distinction, see 
G. Cz\'edli~\cite{gC16}, \cite{gC16a},
\cite{gC17}, and 
G. Cz\'edli, T. D\'ek\'any, G. Gyenizse, and J.~Kulin \cite{CDGK14}.)

Because this paper is concerned with lattices that have a complementary order relation, 
an additional useful formalism is the following.  
An \emph{oriented planar lattice} 
is a triple $(L, <, \gl )$,
where $(L, < )$  is a planar lattice and $\gl$ is a
a specified order relation $\gl$ complementary to $<$.  

In this paper, by an order we mean a \emph{strict} partial ordering, as an irreflexive and transitive relation $<$ 
(which is then necessarily antisymmetric).  
Of course, we can still write $a \leq b$ as needed.  
In the following, we emphasize strictness 
where needed for clarity.
We write $a \sim b$ to indicate that $a$ and $b$ 
are comparable (or equal), 
$a \incomp b$ to indicate incomparability, 
and $a \prec b$ or $b \succ a$ to indicate that $b$ covers $a$.

For basic concepts and notation, we refer to the books \cite{LTF} and \cite{CFL2}.

\subsection*{Acknowlegement}
The authors thank the referee for valuable comments.


\section{Diagrams}\label{S:diagrams}

Let $\RR$ be the ordered set of real numbers 
and let $\RR^2$ be the real plane coordinatized 
by the $x$- and $y$-axes. 
For $u \in \RR^2$, we use the notation $(u_x, u_y)$.

A \emph{diagram} is a finite ordered set $H$ whose elements are distinct points in~$\RR^2$, such that for each covering $p \prec q$ in $H$, 
we have $p_y < q_y$ (the \emph{$y$-isotone} property)
and the line segment $\ol{pq}$ intersects $H$ only in 
the points $p$ and $q$ (the \emph{noninterference} property). 
We write $\Cov H$ for the set of such 
covering line segments  associated with $H$. 
Recall that the order in $H$ 
can be recovered from~$\Cov H$ 
(see, for instance, \cite[Lemma 1]{LTF}).
If $H$ is a lattice, we call $H$ a \emph{lattice diagram}.

The diagram $H$ is \emph{planar} 
if any two distinct segments in $\Cov H$ 
are disjoint or
intersect only at their endpoints. 
A \emph{planar lattice} $L$ 
is a lattice that has a planar lattice diagram $H$, in the sense that $H$ as an ordered set is a lattice isomorphic
to~$L$.

Planar lattice diagrams can have some fragility and some diagrams are ``better drawn'' than others. 
For instance, the planar lattice diagram $L_A$ in Figure~\ref{F:LA} does not have the two properties
we are going to describe now.

A sublattice $S$ of a planar lattice diagram~$L$ 
must again be a planar lattice, 
at least via a diagram with new points 
(see Theorem~\ref{T:planar}),
but we may ask whether $S$ has 
the planarity property ``in place'',
in the sense that line segments directly connecting 
coverings in $S$ intersect only at their end points.
Thus for a planar lattice diagram $L$, 
we define the following property.

\begin{enumerate} 
\item[(Sub)]
For any sublattice $S$ of $L$, reconnecting coverings in $S$ with the appropriate segments 
results in a planar lattice diagram of $S$. 
\end{enumerate}
For example, in $L_A$ of Figure~\ref{F:LA}, the sublattice $S = L_A  -  \set s$ with reconnection $\ol{rw}$
violates the planarity property (Sub).

A planar lattice diagram $L$ is \emph{well drawn} if whenever $p$ is to the left of $q$ in the complementary left-right order, then $p_x < q_x$. 
(That is, the diagram has the left-right $x$-isotone property, in addition to the required $y$-isotone property.) 
In the diagram~$L_A$, the points $w$ and $r$ are on the right boundary but other points have greater $x$-coordinate values.
In contrast, the planar lattice diagram $L_B$ in Figure~\ref{F:LB} is isomorphic to~$L_A$, 
has property (Sub), and is well drawn.
In Section~\ref{S:canonical} we present an algorithmic way to construct diagrams with these properties.

\begin{figure}[htb]
\centerline{\includegraphics{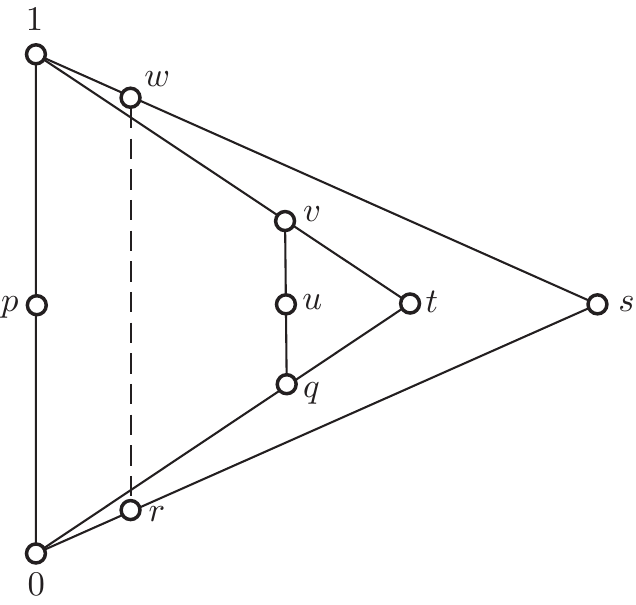}}
\caption{The planar lattice diagram $L_A$}\label{F:LA}
\end{figure}

\begin{figure}[htb]
\centerline{\includegraphics{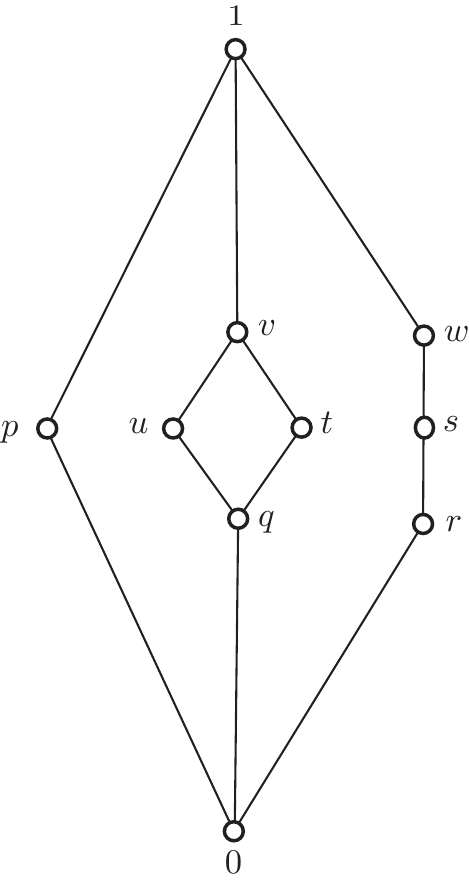}}
\caption{The planar lattice diagram $L_B$, isomorphic to $L_A$}
\label{F:LB}
\end{figure}


\section{The distributive lattice of maximal chains} \label{S:Maximal}

We shall examine left-right comparisons of objects 
within a planar lattice diagram;
we compare maximal chains, elements, and various subsets. 
A basic distinction is whether such a comparison 
can be made directly 
by looking at $x$-coordinates of~points 
(as is the case in comparing two maximal chains 
via their paths) 
or only indirectly 
(as is the case in comparing two lattice elements). 
We begin by introducing
the distributive lattice of maximal chains.

Let $L$ be a planar lattice diagram. 
Let $\MC L$ denote the set of maximal chains of~$L$.
As in Kelly-Rival \cite[p. 641]{KR75},
for $C \in \MC L$, we define the \emph{path function} $f_C$, whose graph $x = f_C(y)$ is
the path along the covering segments of $C$; 
the common domain of all such functions 
is the real interval 
determined by the placement of the diagram in~$\RR^2$. 
We compare members of $\MC L$ via their path functions, writing $C_1 \leq C_2$ if $f_{C_1} \leq f_{C_2}$ 
as functions on a real interval. 
Thus $\MC L$ is an ordered set.

\begin{lemma}
$\MC L$ is a distributive lattice 
with the lattice operations 
determined by the $\max$ and $\min$
of path functions, 
$\max(C_1, C_2)$ \lp the join\rp and $\min(C_1, C_2)$  
\lp the meet\rp .
\end{lemma}

\begin{proof}
It is sufficient to show that the path functions of members of $\MC L$ form
a lattice of functions under the operations of functional $\max$ and $\min$. 
We verify that the set of path functions is closed under these operations.
For two maximal chains $C_1, C_2 \in \MC L$, their path functions may overlap,
but from planarity we know that they can overlap only at elements of $L$ and 
on covering segments they share. 
Therefore, we may take the $\max$ and $\min$ 
of the two path segments between consecutive overlap points 
and patch them together to form two new maximal chains, 
whose path functions consist of one 
having all the mins and the other all the maxes. 
These are the functional $\max$ and $\min$ of~ $C_1$ and $C_2$.
Thus the set of path functions of members of $\MC L$ 
is a sublattice of the distributive lattice 
of all real-valued functions on the appropriate real interval, 
under $\max$ and $\min$.
\end{proof}

Observe that, as in Kelly-Rival \cite[p. 641]{KR75}, 
we use the geometric idea that the two path functions together bound a region (perhaps with multiple components),
with the $\min$ and $\max$ of the functions defining the left and right boundaries.

We can put $\MC L$ to immediate use as follows.
For $p \in L$, the \emph{leftmost maximal chain} at $p$ is 
\begin{align*}
 \LM(p) &= \min\setm{C \in \MC L}{ p \in C},\\
   \intertext{while the \emph{rightmost maximal chain} 
      at $p$ is }
   \RM(p) &= \max\setm{C \in \MC L}{p \in C}.
\end{align*}

\section{The complementary order} 
\label{S:complementary}

Let $L$ be a planar lattice diagram. 
We define an ordering $p \slo q$
of incomparable pairs of elements
that will be the complementary order in Zilber's Theorem
(where ``lr'' in the subscript stands for ``left-right'').
We start with path functions 
as a direct method of left-right comparison 
of elements versus maximal chains
and then develop left-right comparison 
of pairs of elements indirectly. 

For $p = (p_x, p_y) \in L$ and $C \in \MC L$, 
let us say that  
\emph{$p$ is \lp weakly\rp to the left of $C$}, 
in symbols, $p \leqx C$, if $p_x \leq f_C(p_y)$
(where ``pf'' in the subscript
 stands for ``path function''). 
The relations $p \ltx C$, $p \geqx C$, 
$C \leqx p$, and so on, are defined similarly. 

Now to compare two elements $p, q \in L$,
let us write $p \slo q$ if{}f 
the following two conditions hold.
\begin{enumeratei}
\item $p \incomp q$;
\item $p$ and $q$ are separated via the path function
of a maximal chain $C$, 
that is, $p \leqx C \leqx q$.
\end{enumeratei}

The relation $\slo$ will be shown in Theorem~\ref{T:complementary} to be the desired complementary order.  (Other, equivalent definitions of $p \slo q$ are possible, as in~Kelly-Rival~\cite{KR75}, but this
choice, with only a mild condition on the separating maximal chain, is especially easy to apply.)

Our proof of Theorem~\ref{T:complementary} is organized in ``local to global'' fashion.
For a planar lattice diagram $L$, 
if we already have a~complementary left-right order on $L$, then at each element~$p$ we could partition $L$ into several regions: elements strictly to the right of $p$, elements strictly to the left of $p$, and elements comparable to $p$.  
Instead, we proceed in the opposite direction, by first showing the existence of such a partition at each $p$, as
illustrated in~Figure~\ref{F:partition}, and then using this fact to prove the theorem. This framework
provides an efficient way of organizing the facets of reasoning needed to verify the theorem.

\begin{lemma}\label{L:partition}
Let $L$ be a planar lattice diagram and let $p \in L$.
Then $L$ is partitioned into three pairwise 
disjoint regions, 
each with two equivalent descriptions, which we may
call a ``left-right'' description and a counterpart ``path'' description:
\begin{alignat*}{2}
    \Left(p)  &= \setm{q}{q \slo p} ={}\setm{q}{q \ltx \LM(p)},\\
    \Right(p) &= \setm{q}{p \slo q} ={}\setm{q}{q \gtx \RM(p)},\\ 
    \Mid(p)  &=\setm{q}{p \sim q}\hspace{5pt} = \setm{q}{\LM(p) \leqx q \leqx \RM(p)}.
\end{alignat*}
Further, the elements of $\Left(p)$ 
are incomparable with the elements of $\Right(p)$.
\end{lemma}
 
\begin{figure}[thb]
\centerline{\includegraphics{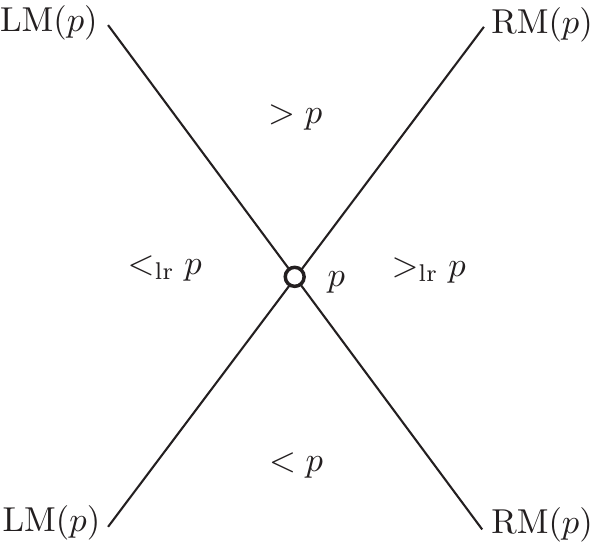}}
\caption{Visualizing the partition at $p$.  
There is a left region, a~right region, and a
middle region consisting of all elements comparable to $p$.  The maximal chains $\LM(p)$ and $\RM(p)$
are contained in this middle region.}
\label{F:partition}
\end{figure}

\begin{proof}
For purposes of the proof, let us consider that $\Left(p)$, 
$\Right(p)$, and $\Mid(p)$ are defined by their left-right descriptions, while the path descriptions define  
$\Leftx(p)$, $\Rightx(p)$, and $\Midx(p)$, respectively,
each of whose equality with its counterpart left-right description remains to be proved.  As a first step, observe that $\Leftx(p)$, $\Rightx(p)$, and $\Midx(p)$ form a partition by their definitions.

In this proof the following two claims will be useful.  
For $C \in \MC L$,
\begin{enumerate}[(A)]
\item if $C \leqx p$, then $C \leq \RM(p)$, or equivalently, $C \ci \Leftx(p) \cup \Midx(p)$;
\item if $C \geqx p$, then $C \geq \LM(p)$, or equivalently, $C \ci \Rightx(p) \cup \Midx(p)$.
\end{enumerate}
(If $p \in C$, then (A) and (B) both apply.)
Indeed, if $C \leqx p$, set $C' = \max(C, \RM(p))$ in $\MC L$; then $p \in C'$
so $C' \leq \RM(p)$, giving $C \leq C' \leq \RM(p)$ in $\MC L$, which is the same as
$C \ci \Leftx(p) \cup \Midx(p)$.  If  $C \geqx p$, a similar argument applies.

Next, to show that $\Midx(p) \ci \Mid(p)$, 
suppose that we have an element $q$ with 
\[
   \LM(p) \leqx q \leqx \RM(p).
\]  
Let $C$ be any maximal chain through $q$ and let
\[
   C' = \max(\LM(p), \min(\RM(p), C)),
\]  
which still has $q$ as an element 
but also satisfies \[\LM(p) \leqx C' \leqx \RM(p).\]
Since $\LM(p)$ and $\RM(p)$ touch at $p$,
we have $p \in C'$ as well, whence $p$ and $q$ are comparable. 
(Here we have applied a familiar lattice-theoretic idea to
the lattice $\MC L$: 
if $a \leq b$ in a lattice, then $x \mapsto a \jj (b \mm x)$ is an isotone map 
of the lattice onto the interval $[a, b]$.)  For the opposite inclusion $\Mid(p) \ci \Midx(p)$, 
if we have $q \sim p$,
then $q$ lies on some maximal chain $C$ through $p$ and hence $\LM(p) \leq C \leq \RM(p)$, giving
$q \in \Midx(p)$.  From the two inclusions 
we obtain that $\Midx(p) = \Mid(p)$.

To show that $\Leftx(p) \ci \Left(p)$, 
let us assume that $q \ltx \LM(p)$.  To attain $q \in \Left(p)$, or equivalently $q \slo p$,
by definition we need to verify the two conditions 
\begin{enumeratei}
\item $q \incomp p$,
\item  $q \leqx C \leqx p$ for some $C \in \MC L$.
\end{enumeratei}

\emph{Ad} (i), 
observe that the opposite, 
$q \sim p$, would imply that $q$ and $p$ 
are on a common maximal chain $C$, 
in which case we would have $C \geq \LM(p)$ 
and hence $q \geqx \LM(p)$, contrary to assumption.
\emph{Ad} (ii), we have $q \leqx C \leqx p$ with $C = \LM(p)$.
Thus $q \in \Left(p)$.  

For the opposite inclusion $\Left(p) \ci \Leftx(p)$, 
let $q \in \Left(p)$.  
We have just assumed that $q \in \Left(p)$, 
which by definition means that $q \slo p$, 
which in turn by definition means that (i) and (ii) hold.  
From (ii), the claim~(A) gives $q \leqx C \leq \RM(p)$,
whence $q \in \Leftx(p) \cup \Midx(p)$, 
while by (i) we have $q \notin \Mid(p)$, 
which we already know equals $\Midx(p)$;
therefore, $q \in \Leftx(p)$ as desired.  
From the two inclusions, 
we obtain that $\Leftx(p) = \Left(p)$.  
A symmetrical argument yields $\Rightx(p) = \Right(p)$.

For incomparability between elements of $\Left(p)$ and elements of $\Right(p)$, consider any two
comparable elements $q, q' \in L$ and let $C$ be a maximal chain including both.  By claims (A) and (B),
\[q, q' \in C \ci \Left(p) \cup \Mid(p)\] or \[q, q' \in C \ci \Right(p) \cup \Mid(p)\] (or both).
In neither case can one of $q, q'$ be  in $\Left(p)$ and the other in $\Right(p)$.
\end{proof}

Now we are ready to assert the ``only if'' direction of Zilber's Theorem.

\begin{theorem}\label{T:complementary}
Let $L$ be a planar lattice diagram.
The relation $\slo$ on $L$ is an order relation,  complementary to the order relation $<$ on $L$.
\end{theorem}

\begin{proof}
The relation $\slo$ is irreflexive by definition, since $p \incomp p$ is false.  
For transitivity, if $q \slo p \slo q'$, 
that is, if $q \in \Left(p)$ and $q' \in \Right(p)$, 
then (i) $q \incomp q'$ by the last clause of Lemma~\ref{L:partition} and (ii) $q \leqx C \leqx q'$ for
the choice $C = \RM(p)$, whence $q \slo q'$.

To verify that $\slo$ is a complement of $<$, we must prove that for $p \neq q$ in $L$, $q$ is comparable
to $p$ by exactly one of $<$ and $\slo$.  But this is the import of the assertion in Lemma~\ref{L:partition} that
the defined regions form a partition.
\end{proof}

The fact that $\slo$ is transitive has an interesting aspect:  If $a \incomp b$ and $b \incomp c$
and we are going to the right each time, then $a \incomp c$.


\section{Three guises of $\RR^2$} \label{S:R2}

In this paper we look at the real plane $\RR^2$ in three ways. The first way, as in Section~\ref{S:diagrams}, 
uses $\RR^2$ as a drawing board for diagrams. Both the set of points of a diagram and its order are our choice, subject only to mild restrictions (the noninterference and the $y$-isotone properties). 

The second guise is $\RR^2$ with the direct-product order. Again, we choose a finite set~$S$ of points of $\RR^2$, but this time they inherit the product order. The ordered set $S$ may lack the $y$-isotone property, but other properties are now automatic.  A~simple observation is the following.

\begin{lemma} \label{L:rectangle}
If $p \prec q$ in $P$, then 
the interval $[p, q]$ in $\RR^2$
\lp in the direct-product order\rp is a rectangle that includes the whole segment $\ol{pq}$. 
\end{lemma}

\begin{proof}
The points of the segment interpolate linearly in $x$ and $y$ between the coordinates of $p$ and $q$.
\end{proof}

\begin{lemma} \label{L:R2-noninterference} If $P$ is a finite ordered subset of $\RR^2$ in the direct-product order, then~$P$ has the \emph{noninterference property:} 
for $p \prec q$ in $P$, the line segment $\ol{pq}$ in $\RR^2$ intersects~$P$ only in $p$ and $q$.
\end{lemma}

\begin{proof}
By Lemma~\ref{L:rectangle}, if $r \in P \cap \ol{pq}$, then $p \leq r \leq q$ in $P$, and then $p \prec q$ implies that $r = p$ or $r = q$.
\end{proof}

\begin{lemma} \label{L:R2-planarity}
If $L$ is a finite ordered subset of $\RR^2$ in the direct-product order and~$L$ is a lattice in its own right \lp not necessarily a sublattice of $\RR^2$\rp, then $L$ has the \emph{planarity property:} two distinct covering line segments intersect only at endpoints, if at all.
\end{lemma}

\begin{proof}
 By Lemma~\ref{L:rectangle}, if $r \in \ol{pq} \cap \ol{p'q'} \ci \RR^2$, where $p \prec q$ and $p' \prec q'$ in $L$ are distinct coverings, then $p \leq r \leq q$ and $p' \leq r \leq q'$ in $\RR^2$, which entails $p \leq q'$ and $p' \leq q$ 
in~$\RR^2$ and hence in~$L$. Then \[p \jj p' \leq r \leq q \mm q'\] in $\RR^2$. Now $p \jj p'$ and $q \mm q'$ are elements of $L$ that lie on both original covering line segments and so are endpoints of those segments. If $p \jj p'$ and $q \mm q'$ are distinct, then together they must be the endpoints of each original line segment. Then the original line segments coincide, contrary to assumption. Otherwise, $p \jj p' = r = q \mm q'$, in which case the intersection point $r$ is in~$L$ and must be an endpoint, as claimed.
\end{proof}

The product order on $\RR^2$ might also be called the ``closed first-quadrant order'', because
$p < q$ if{f} $q - p$ is a point in the closed first quadrant of $\RR^2$ (other than the origin).
A companion complementary order on $\RR^2$ is the ``open fourth-quadrant order'',
similarly defined.

For the third guise, we again put an order on all of $\RR^2$,  this time by declaring that $p < q$ when $p \neq q$ and the vector $q - p$ makes an angle of at most $45^\circ$ (in absolute value) with the positive $y$-axis.  In other words, we have rotated the product order by $45^\circ$ counterclockwise in terms of which points are greater than the origin.  Let us call this guise of $\RR^2$ the \emph{slanted order}.  Observe that the slanted order has a~complementary order $\slo$, 
namely $p \slo q$ if $p \neq q$ and the vector $q - p$ makes an angle of less than $45^\circ$ with the positive $x$-axis.

While we could transform between the product order and the slanted order
by a simple $45^\circ$ rotation,
it is ultimately more helpful to use a rational map that is a scaled rotation.  Accordingly, to transform
from the slanted order to the product order we shall use
$\alpha: \RR^2 \rightarrow \RR^2$ given by
\[
   \ga \colon (x, y) \mapsto (x + y, y - x),
\]
which is a $45^\circ$ clockwise rotation scaled by $\sqrt 2$.  Its inverse $\gb = \ga^{-1}$ then takes
us from product order to slanted order and is given by
\[
   \gb \colon (x, y) \mapsto (x - y, x + y)/2.
\]

\begin{lemma} \label{L:R2-product-rot}
If $L$ is a finite ordered subset of~$\RR^2$ 
in direct-product order such that $L$ is a lattice in its own right, then
$\gb(L)$ is a planar lattice diagram isomorphic to $L$, with slanted order.
\end{lemma}

\begin{proof}
By Lemmas \ref{L:R2-noninterference} and \ref{L:R2-planarity}, $L$ has at least the properties of noninterference and planarity, properties that are preserved by $\gb$.
It remains to verify that~$\gb(L)$ has the $y$-isotone property.
Indeed, if $p' \prec q'$ and $p$ and $q$ are such that $p' = \gb(p)$ and $q' = \gb(q)$,
then $p \leq q$ in the product order, so $p_x \leq q_x$ and $p_y \leq q_y$
(with one of these strict); then $p'_y = \frac 12 (p_x + p_y) < \frac 12 (q_x + q_y) = q'_y$,
as required.
\end{proof}

\begin{lemma} \label{L:R2-slanted-planarity}
If $L$ is a finite ordered subset of $\RR^2$ in slanted order and $L$ is a lattice in its own right,
then $L$ is a planar lattice diagram.
\end{lemma}

\begin{proof}
Let $L' = \ga(L)$.  Then
Lemma~\ref{L:R2-product-rot} applies to $L'$ (in place of $L$) to show that
$L = \gb(\ga(L))$ is a planar lattice diagram in slanted order.
\end{proof}


\section{Proof of Zilber's Theorem Plus}\label{S:Zilber-proof}

In this section we prove Zilber's Theorem Plus, and hence Zilber's Theorem itself. We start with additional lemmas.
The first lemma is due to B. Dushnik and E.\,W. Miller \cite[Lemma 3.51]{DM41}.

\begin{lemma} 
\label{L:ii-to-iii} 
If $\gr$ and $\gl$ are complementary \lp strict\rp order relations on a set $P$, then $\gr \uu \gl$ and 
$\gr \uu \tilde{\gl}$ are both total \lp strict\rp order relations on $P$, 
where $\tilde{\gl}$ is the dual of~$\gl$, and their intersection is $\gr$.
\end{lemma}

\begin{proof}
Write $\theta = \gr \uu \gl$. 
Since $\gr$ and $\gl$ are complementary, 
$\theta$ is total and irreflexive.  
It will be useful to verify that $\theta$ is antisymmetric, which we must do directly, since transitivity is not yet available.  
Indeed, if $p \rtheta q \rtheta p$ and if both 
instances of $\theta$ involve the same one of $\gr$ and $\gl$, 
then antisymmetry of an order relation applies, 
while if the two instances involve both $\gr$ and $\gl$, 
then complementarity applies, a contradiction either way.  
Next we verify that $\theta$ is transitive, 
by using antisymmetry:
If $p \rtheta q \rtheta r$ but not $p \rtheta r$, 
then $r \rtheta p$, making a ``triangle'' 
\[
   p \rtheta q \rtheta r \rtheta p.
\]
But then two consecutive sides of the triangle 
must both involve $\gr$ or both involve $\gl$, 
say $q \rgl r \rgl p$, 
yielding $q \rgl p$ and thence $q \rtheta p$, 
which by antisymmetry of $\theta$ contradicts 
that $p \rtheta q$; 
therefore, $p \rtheta r$ as desired.
Finally, since $\gr$ and $\gl$ are complementary,
\[(\gr \uu \gl) \cap (\gr \uu \tilde{\gl}) = \gr \uu (\gl \cap \tilde{\gl}) = \gr \uu \emptyset = \gr.\qedhere\]
\end{proof}

The next lemma is due to T. Hiraguti \cite[Theorem 9.3]{tH55}; we state it in a special case only.

\begin{lemma}\label{L:iii-to-iv}
Let $\gr$ be the order relation of an ordered set $P$  and let $\tau_1, \tau_2$ be total orders on~$L$ with $\gr = \tau_1 \cap \tau_2$. Then the diagonal map 
\[
   \gd \colon (L, \gr) \mapsto (L, \gt_1) \times (L, \tau_2)
\]
defined by $\gd(p) = (p, p)$ is an order embedding with respect to the direct-product order.
\end{lemma}

\begin{proof}
$\gd(p) < \gd(q)$ if{}f $(p, p) < (q, q)$ coordinatewise in the product
iff $p \rgt_1 q$ and $p \rgt_2 q$ iff $p \mathbin{(\rgt_1 \cap \rgt_2)} q$ iff  
$p \rgr q$ in $L$.
\end{proof}

\begin{proof} [Proof of Zilber's Theorem Plus] \hfill

(i) implies (ii).  The hypothesis is that $L$ \emph{has} a planar lattice diagram.
Without loss of generality,
we may assume that $L$ \emph{is} a planar lattice diagram.  Then Theorem~\ref{T:complementary}
applies to show the existence of a complementary order.

(ii) implies (iii). By Lemma~\ref{L:ii-to-iii}.

(iii) implies (iv). By Lemma~\ref{L:iii-to-iv}.

(iv) implies (i). Given an order embedding of $L$ 
into the direct product of two chains, 
further embed the chains into $\RR$ 
in any order-preserving way, 
in which case~$L$ is embedded in $\RR^2$ 
with direct-product order, 
and then apply the transformation 
$\gb$ of Lemma~\ref{L:R2-product-rot}. 
By the same lemma, 
the resulting transformed image $\bar L$ of $L$ 
is a planar lattice diagram isomorphic to $L$.
\end{proof}


\section{Planar canonical form}\label{S:canonical}

The converse implication in Zilber's Theorem 
states that a finite lattice with a complementary order relation is planar. Although this fact has been proved by the combined implications (ii) to (iii) to (iv) to (i) in the proof of Theorem~\ref{T:Zilber-plus}, there is a more direct description in terms of a ``canonical'' planar lattice diagram, as follows.

We need some more notation. For any finite ordered set $(P, <)$ and $p \in P$, let us write 
\begin{align*}
\Lo(p, <) &= \setm{q \in P}{q < p},\\
\Hi(p, <) &= \setm{q \in P}{p < q}. 
\end{align*}

We define the \emph{imbalance} of $<$ at $p$ by
\[
   \imbal(p, <) = |\Lo(p, <)| - |\Hi(p, <)|.
\]
\noindent
Thus elements low with respect to $<$ may have negative imbalance, while higher elements have positive imbalance. For example, if $P$ is a planar lattice diagram $L$, 
then we obtain 
\[
   \imbal(p, <) = |\Under(p)| - |\Over(p)|,
\] 
where 
\begin{align*}
   \Under(p) &= \setm{q}{q < p};\\
    \Over(p) &= \setm{q}{q > p}
\end{align*}
For the complementary relation $\slr$, we have 
\[
   \imbal(p, \slr) = |\Left(p)| - |\Right(p)|.
\]

We shall show that any planar lattice diagram, or more generally any finite lattice with specified complementary order relation, is isomorphic to a \emph{canonical} planar lattice diagram in which \emph{positions of elements} are determined by \emph{counts of relationships}, as follows.

We say that a planar lattice diagram $M$ 
is in \emph{planar canonical form}, 
if $M$ has the slanted order on $\RR^2$ 
and for each $p \in M$, we have
\[
   p = \lp \imbal(p, \slo), \imbal(p, <) \rp.
\]
For a planar lattice $L$ with specified complementary relation $\gl$, we say that the \emph{canonical map} $\Canon\colon L \rightarrow \RR^2$ in slanted order is defined by
\[
	\Canon(p) = \lp \imbal(p, \gl), \imbal(p, <) \rp.
\]
Thus in the terminology of oriented planar lattices (Section~\ref{SS:Terminology}), the oriented planar lattice diagram $(M, <, \slo)$ in slanted order is in planar canonical
form iff $\Canon$ is the identity map on $M$.  We may write $\lp \Canon(L), <, \slo \rp$ for
$\Canon(L)$ as an oriented planar lattice.

\begin{theorem} \label{T:canonical}
Let $L$ be a finite lattice with complementary order relation~$\gl$ and let $\Canon$ be the associated canonical map of $L$ into $\RR^2$.  Let $\bar L = \Canon(L)$ with slanted order.
Then
\begin{enumerate}
\item[\tup{(1)}] $\bar L$ is a planar lattice diagram in planar canonical form.

\item[\tup{(2)}] $\Canon$ is an isomorphism of  $\lp L, <, \gl \rp$ with $\lp \bar L, <, \slo \rp$ as oriented planar lattices.
\end{enumerate}
\end{theorem}

\begin{proof}
We claim that $\Canon$ is achievable as a composition of lattice isomorphisms
involved in the steps of the proof of Theorem~\ref{T:Zilber-plus}. The claim is justified as follows.
For clarity, let us write $\gr$ for $<$.  
Lemma~\ref{L:ii-to-iii} shows that $\gr \cup \gl$ and $\gr \cup \tilde{\gl}$ are total orders on~$L$, making $L$ a chain in two ways.   
Then Lemma~\ref{L:iii-to-iv} order-embeds $L$ in the direct product of these two chains by $\gd \colon p \mapsto (p, p)$.
 
Next, as in the proof of ``(iv) implies (i)'', 
we have to to choose embeddings of the two chains into $\RR$; let us use $\imbal$ as the map in each case.  
For the first chain, we have 
\[
   \imbal(p, \gr \cup \gl) = \imbal(p, \gr) + \imbal(p, \gl),
\]
since the relations $\gr$ and $\gl$ are disjoint.  
Similarly, taking into account the dualization of $\gl$, 
we get 
\[
   \imbal(p, \gr \cup \tilde{\gl}) = \imbal(p, \gr) - \imbal(p, \gl).
\]
In other words,
if we write $a = \imbal(p, \gr)$ 
and $b = \imbal(p, \gl)$, so far we have a map of~$L$
into $\RR^2$ in product order given 
by $p \mapsto \lp a + b, a - b \rp$, which we may
recognize as $\ga(b, a)$, where $\ga$ is the scaled rotation defined
in Section~\ref{S:R2}.  Now as a final
step we can transform to slanted order by composing with $\gb = \ga^{-1}$ to obtain
$(b, a)$, which is $\Canon(p)$.
This completes the proof of the claim and also shows
that $\bar L$ is indeed a planar lattice diagram for $L$, as asserted in (1).

Now let us examine how the complementary order $\gl$ has been treated under this
composition.  Because of the dualization in the second total order, $\gl$ is first mapped
via $\gd$ and the imbalance embeddings to
the open-fourth-quadrant order of Section~\ref{S:R2}, which is then scaled and
rotated by $\gb$ to the complementary order $\slo$ on $\RR^2$ in slanted order, at each
step restricted to the appropriate image of $L$.  This completes the proof of (2).

Finally, to complete the proof of (1), we must show that $\bar L$ has planar canonical
form.  
But from (2) we know that the counts defining the relevant imbalances are preserved by $\Canon$.
Therefore $\bar L$
has the same imbalances as $L$ and hence $\Canon$ is the identity map on $\bar L$ as required.
\end{proof}

\begin{corollary} \label{C:unique}
Each isomorphism class of oriented planar lattices contains a unique member in planar canonical
form.
\end{corollary}

\begin{proof}
Each isomorphism class contains at least one member in planar canonical form, since if
$\lp L, <, \gl \rp$ is one member of the class then $\lp \Canon(L), <, \slo \rp$ is in planar
canonical form, by Theorem~\ref{T:canonical}.
Since an isomorphism of oriented planar lattices preserves the counts
needed to compute the relevant imbalances, each isomorphism class has at most one member in
planar canonical form.
\end{proof}

The uniqueness in Corollary~\ref{C:unique} justifies the appellation ``canonical''.

Now let us focus on the application where $L$ is a planar lattice diagram and $\gl$ is $\slo$.
Here applying
$\Canon$ provides  an automated way of ``straightening out'' the diagram~$L$.
For example, the planar lattice diagram $L_A$ of Figure~\ref{F:LA} and its isomorph~$L_B$ of Figure~\ref{F:LB} share the canonical form $L_C$ shown in Figure~\ref{F:LC}.  Consider
the element $t \in L_A$, which has left-right imbalance $2 - 3 = -1$ 
(since $p,u \slo t \slo r, s, w$) and down-up imbalance $2 - 2 = 0$, giving the point $(-1,0)$. This calculation explains why the element labeled $t$ in Figure~\ref{F:LC} is to the left of center (even though Figure~\ref{F:LB} has $t$ right of center);  there are more elements to the right of $t$ than to the left.  Observe also that the lattice elements 0 and 1 are necessarily on the $y$-axis in any
example.

\begin{figure}[htb]
\centerline{\scalebox{1}{\includegraphics{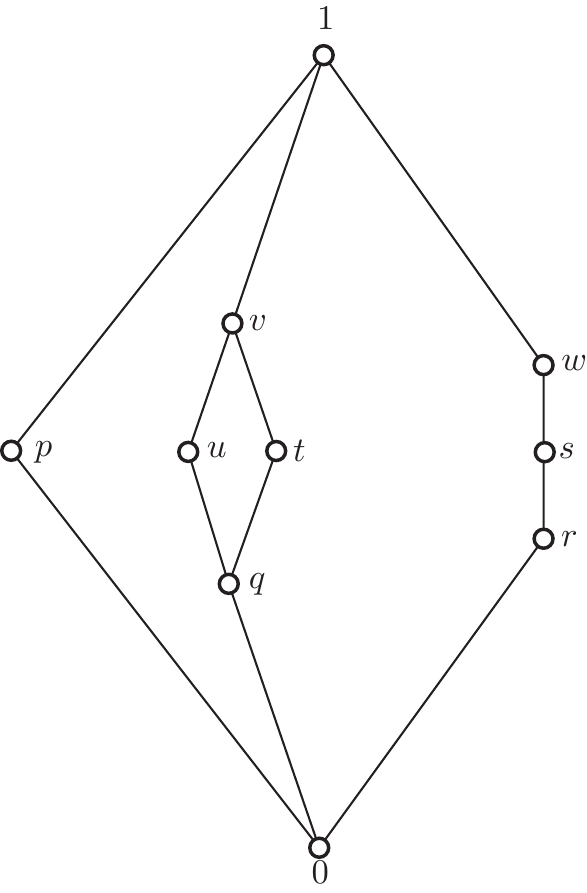}}}
\caption{The planar lattice diagram $L_C = \Canon(L_A) = \Canon(L_B)$}
\label{F:LC}
\end{figure}

We state two more properties of the planar canonical form.
\begin{lemma} \label{L:canon-sub-in-place}
A planar lattice diagram $L$ in canonical form has the  property \tup{(Sub)}.
\end{lemma}

\begin{proof}
Apply Lemma~\ref{L:R2-slanted-planarity} to sublattices, noting that a lattice in planar canonical form has slanted order by definition.
\end{proof}

\begin{lemma} \label{L:canon-well-drawn}
A planar lattice diagram in canonical form is well drawn.
\end{lemma}

\begin{proof}
If $p \slo q$, then 
\[
p_x = \imbal(p, \slo) < \imbal(q, \slo) = q_x,
\]
since $q$ has more elements preceding it with respect 
to $\slo$ than does $p$.
\end{proof}

Several other types of canonical diagrams have been defined for planar semimodular lattices in 
G. Cz\'edli~\cite{gC17}.
An interesting uniqueness statement is  \cite[Proposition~5.14]{gC17}.


\section{Coverings in the left-right order}\label{S:left-right-coverings}

Now we proceed to examine some ramifications of the proof of Zilber's Theorem Plus in more detail, starting with 
the implication ``(i) implies (ii)''.

In a planar lattice diagram $L$, 
let us say that $p$ is 
\emph{immediately to the left of} $q$,
if~$p$ is covered by $q$ in the $\slo$ order, in formula,
$p \preclr q$. How can we characterize this
relation?

We can begin by applying Lemma~\ref{L:partition}, which gives us the schematic portrayal 
in Figure~\ref{F:double-partition}.  It is evident that $p \preclr q$ precisely when the space
strictly between $\RM(p)$ and $\LM(q)$  has no element of the diagram.
Observe that the two boundaries intersect 
at $p \jj q$ and $p \mm q$.  Note, however, that the maximal chains portrayed can
have additional, unrelated
intersections not shown; for example, all maximal chains in $L$ include 0 and 1.
Therefore in the following we restrict attention explicitly to $[p \mm q, p \jj q]$.

\begin{figure}[bht]
\centerline{\includegraphics{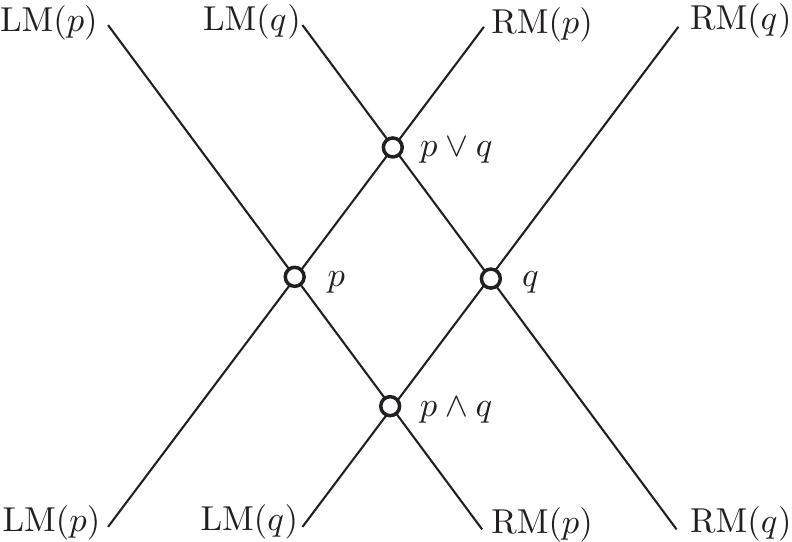}}
\caption{Partitions at $p$ and at $q$ together for $p \preclr q$} 
\label{F:double-partition}
\end{figure}

For incomparable elements $p$, $q$ 
of a planar lattice diagram $L$, 
let us say that the \emph{local region} 
determined by $p$ and $q$ is 
\[
   \LR(p,q) = 
   \setm{z \in [p \mm q, p \jj q]}
         {\RM(p) \leqx z \leqx \LM(q)}.
\]  
The \emph{interior} of such a local region is 
\[
   \Int(p,q) = \setm{z \in \LR(p,q)}
                    {\RM(p) \ltx z \ltx \LM(q)}.
\]
The \emph{left boundary} of such a region is $\LR(p,q) \cap \RM(p)$ and the
\emph{right boundary} is $\LR(p,q) \cap \LM(q)$.  

\begin{figure}[b]
\centerline{\scalebox{.8}{\includegraphics{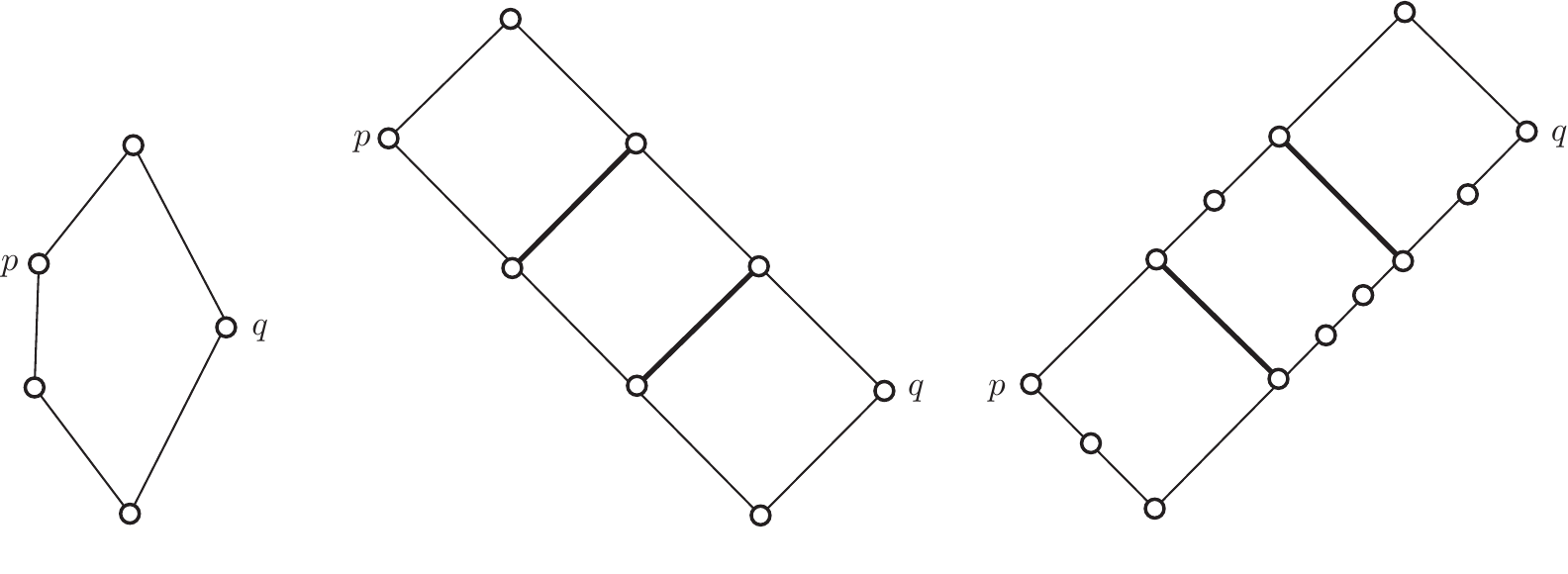}}}
\caption{A cell, a (leftward) ladder, and a (rightward) e-ladder from $p$ to $q$, as planar lattice diagrams} 
\label{F:ladders}
\end{figure}

\begin{figure}[b]
\centerline{\scalebox{.8}{\includegraphics{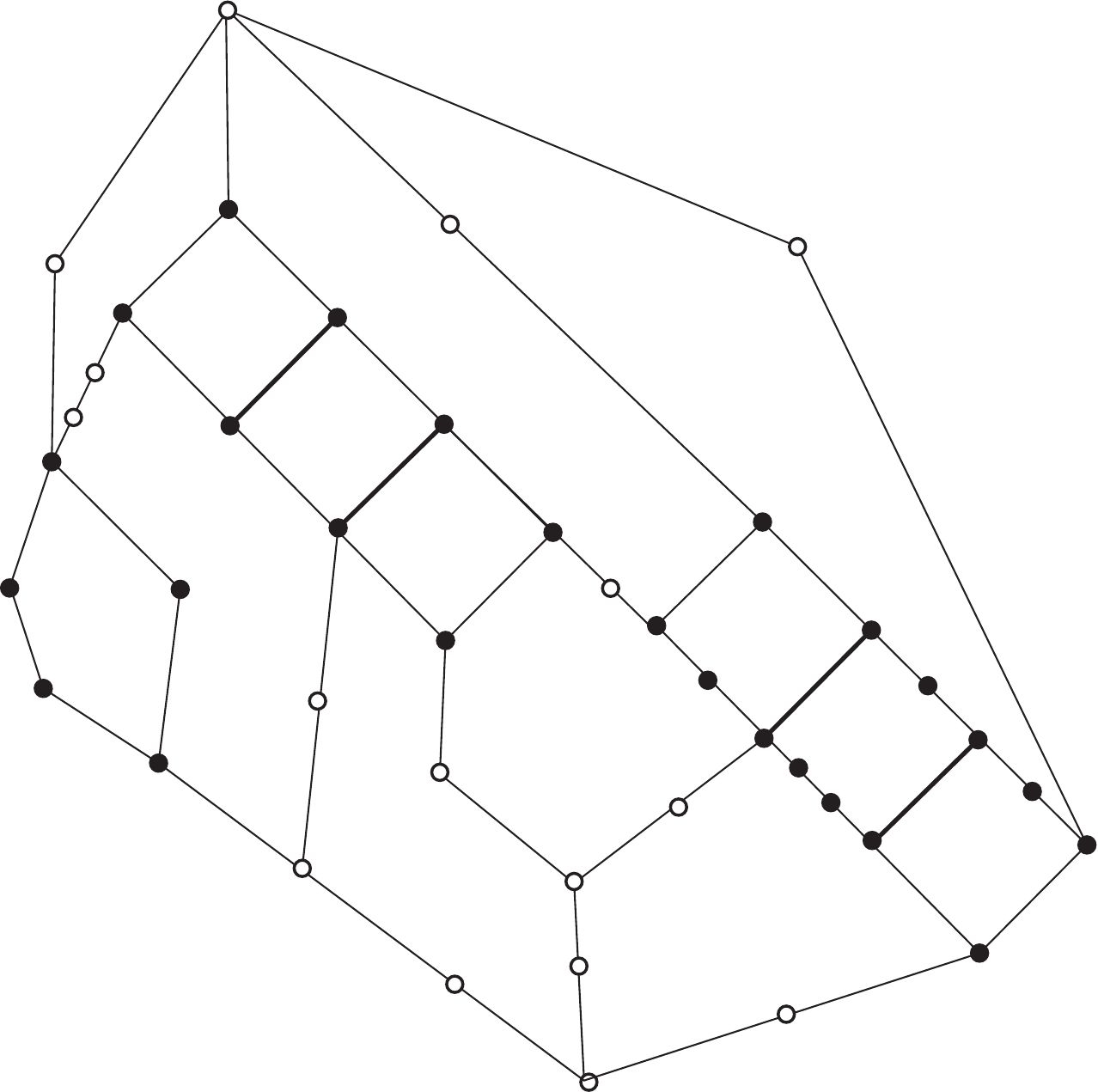}}}
\caption{A cell, a ladder, and an e-ladder as 
local regions in a planar lattice diagram} 
\label{F:laddersnew}
\end{figure}

In a planar lattice diagram $L$,
a~\emph{ladder} from $p$ to $q$ 
is a local region $\LR(p,q)$
isomorphic to $\SC 2 \times \SC n$ for some $n \geq 2$, 
where $\SC n = \set{0, 1, \dots, n-1}$  
(with the natural ordering).
More generally, an~\emph{e-ladder} 
(\emph{extended ladder}) from $p$ to $q$ 
is a local region $\LR(p,q)$ 
consisting of an isomorph of $\SC 2 \times \SC n$ 
together with possible additional boundary elements.  
Thus an e-ladder has empty interior.
A \emph{rung} of a ladder or e-ladder 
is a covering between two elements of opposite boundaries
(not including top or bottom elements).
In Figures~\ref{F:ladders} and~\ref{F:laddersnew} 
the rungs are indicated with bold lines.
We say that an e-ladder is \emph{leftward}, 
if its rungs have positive slant or
\emph{rightward} if its rungs have negative slant, 
as indicated in Figure~\ref{F:ladders}.  
A~\emph{cell} is an e-ladder with no rungs 
and is considered to be both leftward and rightward.

\begin{theorem} \label{T:lr-coverings}
For incomparable elements $p$, $q$ of a planar lattice diagram $L$, the following conditions are equivalent.
\noindent
\begin{enumeratea}
\item $p \preclr q$;
\item  $\Int(p,q) = \es$;
\item $\LR(p,q)$ is an e-ladder from $p$ to $q$.
\end{enumeratea}
\end{theorem}

\begin{proof}

We have already observed the equivalence of (a) and (b)  
and noted that (c) implies~(b). 
It remains to verify that (b) implies (c).  
Accordingly, assume that $\LR(p,q)$ has empty
interior, whence it is the union of its left and right boundary chains, which intersect only
at $p \mm q$ (the bottom) and $p \jj q$ (the top).  
Consider a covering between $\ell$ and $r$,
where $\ell$ is on the left boundary and $r$ on the right, neither one being the top or bottom.
We have either Case 1: $\ell \prec r$ (where $\ol{\ell r}$ has positive slope) 
or Case~2: $\ell \succ r$ (where $\ol{\ell r}$ has negative slope).

Let us assume Case 1.  
What configurations involving
$\ell, r, p, q$ are possible, 
taking into account that $\ell \sim p$ and $r \sim q$?  
If $p \leq \ell$
and $q \leq r$, then $p \leq \ell < r$ so $p \jj q \leq r$, whence $r$ is at the top, 
which is excluded.  
Similarly, we cannot have $p \geq \ell$ and $q \geq r$.  
If $p \leq \ell$ and $q \geq r$,
then $p \leq \ell < r \leq q$, 
contradicting $p \incomp q$.  
Therefore the only viable
possibility is that $p > \ell$ and $q < r$.
  
Similarly, Case 2 yields that $p < \ell$ and $q > r$.

Next, consider two coverings, 
$\ell$ with $r$ and $\ell'$ with $r'$.  
Could they be of opposite cases, say
$\ell \prec r$ and $\ell' \succ r'$?  
We claim not.  To have opposite cases would entail
$\ell < p < \ell'$ and $r > q > r'$ and then
$\ell, r'$ would be jointly below $\ell', r$, giving $\ell \leq \ell \jj r' \leq \ell' \mm r \leq r$;  then since
$\ell \prec r$, either (i) $\ell = \ell \jj r'$ or (ii) $\ell' \mm r = r$.  If~(i), then $r' \leq \ell < \ell'$ and
since $r' \prec \ell'$ we would have $r' = \ell$, 
which is not possible since $\ell$ and $r'$ are on opposite
boundaries.  Similarly, (ii) fails, proving the claim.
Thus all side-to-side coverings share the same case.

In the following assume that the shared case is Case 1.
Now observe that the coverings are consistently sequenced, 
in the sense that if
$\ell \prec r$ and $\ell' \prec r'$, 
then $\ell < \ell'$ is equivalent to $r < r'$; 
otherwise, say if $\ell < \ell'$ but $r \geq r'$, 
we would have $\ell < \ell' < r' \leq r$ 
in contradiction to $\ell \prec r$.  
Therefore we can collect the coverings 
as $\ell_j \prec r_j$, $j = 1,\ldots,m$, 
with $\ell_1 < \ell_2 < \cdots < \ell_m$ 
and $r_1 < r_2 < \cdots < r_m$, 
which is all that is needed to
show that $\LR(p,q)$ is an e-ladder 
with rungs $\ell_j \prec r_j$.
\end{proof}

\begin{corollary} \label{C:successive-ladders}
For $p, q$ in a planar lattice diagram $L$, 
the relation $p \slo q$ holds if{}f 
there is a sequence of elements 
$p_0 = p, p_1, \ldots, p_k = q$ in which,
for each two consecutive members $p_i, p_{i+1}$,
there is an e-ladder of $L$ from $p_i$ to $p_{i+1}$.
\end{corollary}


\section{Traversals of a planar lattice diagram}\label{S:traversals}

Another important aspect of Zilber's Theorem Plus
is ``(ii) implies (iii)''.
We~start by assuming a complementary order and consider the associated two total orders of Lemma~\ref{L:ii-to-iii}.
For each of these total orders, what is the actual sequence of coverings encountered as we traverse the lattice in the  direction of increasing order?

Let $L$ be a planar lattice diagram.
We can define the \emph{up-right total order}
on $L$ as the union of the order $<$ of $L$ 
and the complementary order $\slo$; 
by~complementarity, 
this union is a total order (B. Dushnik and E.\,W. Miller \cite[Lemma~3.51]{DM41} and Lemma~\ref{L:ii-to-iii} above).
The name reflects the fact that  as we traverse this order, 
we move either up via $<$ on $L$ 
or to the right via $\slo$. Equivalently, the \emph{up-right traversal} of~$L$ is the listing of~the elements 
of~$L$ in increasing up-right total order. 
(A priori, it is not clear that such a traversal always exists.)  The following theorem presents an algorithm for computing this traversal. For convenience, let us say that a covering segment for $p \prec q$ is a \emph{dangling edge} of $q$.

\begin{theorem} \label{T:ur-traversal}
Let $L$ be a planar lattice diagram.
\noindent
\begin{enumerate}
\item[\tup{(1)}] If for every element $q \in L$ 
that covers more than one element, we
delete all dangling edges of $q$ except the rightmost, then the remaining covering edges of $L$ 
form a directed spanning tree $\Tur$ of $L$, 
rooted at $0$.
\item[\tup{(2)}] The up-right traversal of $L$ is the standard depth-first traversal of the tree $\Tur$, going secondarily left to right. In other words, starting at $0$, we must
repeatedly look for the leftmost branch not yet visited, 
and follow it from its lowest unvisited element 
to its end.
\end{enumerate}
\end{theorem}

\begin{proof}
\emph{Ad} (1).  For each $p \in L$ with $p > 0$,
we have retained all the covering line segments 
needed to compute $\RM(p)$ 
from $p$ down to $0$. 
Then by the same token,
there is an up-directed path from $0$ to $p$. 
Therefore $\Tur$ spans $L$. 
Because each node of $\Tur$ other than $0$ 
has in-degree $1$, it follows that $\Tur$ is a tree.

\emph{Ad} (2). 
In the depth-first traversal of $\Tur$, at each node $p$, 
we either move up via $<$ or jump to a new branch 
starting with a $p'$ that is strictly to the right of $\RM(p)$, 
in which case we have $p \slo p'$
by Lemma~\ref{L:partition}.
\end{proof}

An example of the spanning tree for the up-right traversal is given in Figure~\ref{F:T-UR}. The elements are labeled sequentially from 0 in traversal order.

\begin{figure}[thb]
\centerline{\scalebox{1}{\includegraphics{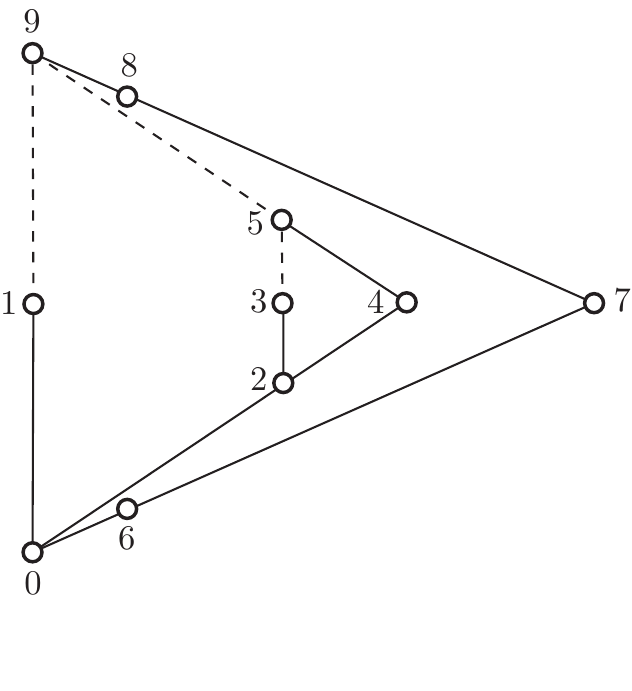}}}
\caption{The spanning tree $\Tur$}
\label{F:T-UR}
\end{figure}

\begin{figure}[ht]
\centerline{\scalebox{1}{\includegraphics{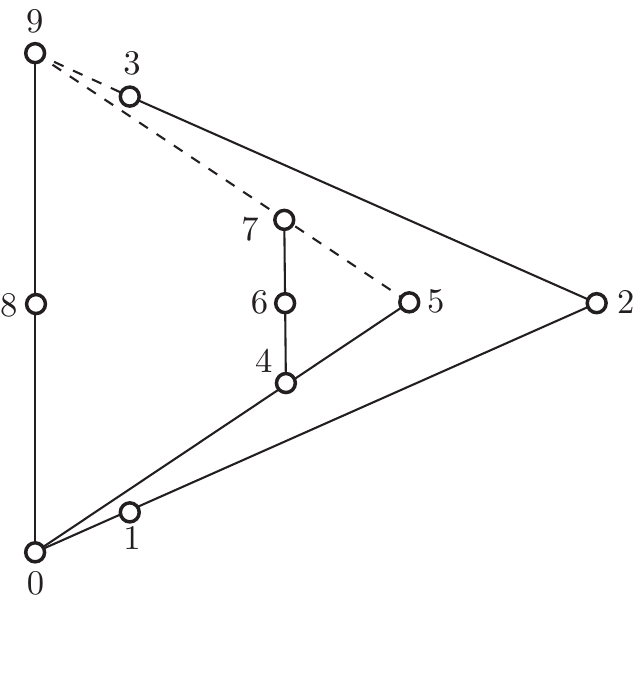}}}
\caption{The spanning tree $\Tul$}
\label{F:T-UL}
\end{figure}

Why does the up-right traversal
involve retaining only the \emph{rightmost} covering 
under each element $p \in L$,
rather than the leftmost?  
Since the up-right total order is consistent with $<$, we will not visit $p$ until all elements \emph{under} $p$ have already been visited, and that happens only when all elements \emph{covered} by $p$ have been visited, the rightmost of these being the last.

Similarly, the \emph{up-left total order} 
on $L$ is the union of the order~$<$ of $L$ and the \emph{dual} 
$\widetilde{\slo}$ of $\slo$.
By considering the vertical reflection of the lattice diagram $L$ and following the $\Tul$ construction,
we obtain a new spanning tree $\Tul$. 
Figure~\ref{F:T-UL} shows 
an example of the up-left traversal.

We also know from the definition of a complementary order that the intersection of the up-right total order and the up-left total order is the original order $<$ on $L$. 


\section{Other ways of defining the left-right order for elements}\label{S:left}
Working with planar lattices and their diagrams, as a rule
there is little doubt visually as to when one element is 
to the left of another element. A convenient mathematical definition was presented in Section~\ref{S:complementary}. Here we mention two alternative approaches. Naturally, both lead to the same left-right order $\slo$.

\subsection{The Kelly-Rival approach}\label{SS:K-R}
Kelly-Rival \cite{KR75} used as a starting point the fact that 
that for the covers of a single element $u$,   
the assignment of a left-to-right order is obvious;
for a cover $c$ of an element $u$, 
let $u \measuredangle c$ denote 
the angle the segment $\ol{uc}$ 
makes with the negative $x$-axis.
Then we can order the covers in sequence 
by these angles.
Note that this order does not necessarily correspond  
to a comparison of $x$-coordinates.

This order can be extended
to arbitrary $a \incomp b$ in $L$
by defining $a$ to be to the left of the element $b$
if $a \mm b$ has covers $a' \leq a$ and  $b' \leq b$ 
such that $a'$ is to the left of the element $b'$ 
as covers of $a \mm b$, 
as defined in the previous paragraph.
Our left-right definition generalizes the condition of
Proposition 1.6  in Kelly-Rival \cite{KR75}.

\subsection{An approach via sequences of e-ladders} \label{SS:ladders} 
Starting from cells as a basic structure, one might think that for $a, b$ in a lattice diagram $L$, 
the relation $a \slo b$ holds if{}f $a$ and~$b$ are connected by a sequence of cells; more specifically, $a = a_0,\ldots, a_k = b$, where for each $i$ with $0 < i \leq k$,  there is a cell $S$ with $a_{i-1}$ on the left boundary and~$a_i$ on the right boundary of $S$.  But any ladder with a rung is a counterexample; 
the two doubly irreducible elements are not so connected.  Instead, Corollary~\ref{C:successive-ladders} suggests that appropriate connections are not via cells alone but rather via e-ladders.

Setting aside the theory of $\slo$ for a moment, 
we can develop an equivalent concept by defining
the binary relation $p \eloi q$
in a planar lattice diagram~$L$ as follows. 
Let $p \eloi q$ if there is an e-ladder from $p$ to $q$.  
(Note that the definition of an e-ladder
in Section~\ref{S:left-right-coverings} 
does not depend on a knowledge of $\slo$.)
Define the relation~$\elo$ to be 
the transitive extension of $\eloi$.  
Now we can connect the two versions 
of left-right order by applying
Corollary~\ref{C:successive-ladders}, 
to obtain the following.

\begin{theorem}\label{T:identical}
In a planar lattice diagram $L$, 
the relations $\elo$ and $\slo$ are the same.
\end{theorem}

By utilizing this result, the property of being well drawn can be easily checked:

\begin{corollary} \label{C:elo-well-drawn}
A planar lattice diagram $L$ is well drawn if{}f 
for $p,q \in L$, the inequality
$p_x < q_x$ holds whenever $p \eloi q$.
\end{corollary}

Or equivalently:

\begin{corollary} \label{C:elo-well-drawn2}
A planar lattice diagram $L$ is well drawn if{}f 
all e-ladders in $L$ are well drawn.
\end{corollary}


\section{Ordered subsets and quotient lattices}\label{S:Quotient}

In this section, we prove some immediate consequences 
of Zilber's Theorem Plus.
The following result---in a slightly different form---can be
found in R. Nowakowski, I. Rival, and J. Urrutia \cite{NRU92},
published in 1992.

\begin{theorem}\label{T:planar}
Let $L$ be a planar lattice. 
Let $K$ be a finite lattice with an order embedding into $L$.
Then $K$ is a planar lattice.
\end{theorem}

\begin{proof}
By Zilber's Theorem Plus, 
$L$ has order dimension at most $2$. 
Since $K$ has an order embedding into $L$, 
it follows that $K$ also has order dimension at most $2$.
Again, by Zilber's Theorem Plus, $K$ is planar. 
\end{proof}

\begin{theorem}\label{T:join-congruence}
Let $L$ be a planar lattice and let $\bgg$ 
be a join-congruence of $L$.
Then~$L/\bgg$ is also a planar lattice.
\end{theorem}

\begin{proof} 
The quotient $L/\bgg$ is a finite join-semilattice with zero and is therefore a lattice.
Since the top elements of the $\bgg$-blocks provide an order embedding of~$L /\bgg$ into $L$, Theorem~\ref{T:join-congruence} applies.  (In fact, this embedding is known to be join-preserving.)
\end{proof}

By duality, we obtain the following statement.

\begin{corollary}\label{C:meet-congruence}
Let $L$ be a planar lattice and let $\bgg$ 
be a meet-congruence of $L$.
Then $L/\bgg$ is also a planar lattice.
\end{corollary}

And as a special case, we get the following.

\begin{corollary}\label{C:congruence}
The quotient lattice of a planar lattice by a lattice congruence
is also a planar lattice.
\end{corollary}

\end{document}